\documentclass{amsart}

\usepackage[ruled,vlined,linesnumbered]{algorithm2e}

\usepackage{graphicx}
\makeatletter
\DeclareFontFamily{U}{tipa}{}
\DeclareFontShape{U}{tipa}{bx}{n}{<->tipabx10}{}
\newcommand{\arc@char}{{\usefont{U}{tipa}{bx}{n}\symbol{62}}}%

\newcommand{\arc}[1]{\mathpalette\arc@arc{#1}}

\newcommand{\arc@arc}[2]{%
  \sbox0{$\m@th#1#2$}%
  \vbox{
    \hbox{\resizebox{\wd0}{\height}{\arc@char}}
    \nointerlineskip
    \box0
  }%
}
\makeatother

\usepackage{marvosym}

\usepackage{pstricks, pst-node, pst-circ, pst-plot, pst-3dplot, pst-solides3d, pstricks-add, pst-eucl} 

\usepackage{picins,graphicx}
\usepackage{paralist}
\usepackage{color}

\usepackage[verbose]{wrapfig}

\usepackage{subfigure}
\renewcommand{\thesubfigure}{\thefigure.\arabic{subfigure}}
\makeatletter
\renewcommand{\p@subfigure}{}
\renewcommand{\@thesubfigure}{\thesubfigure:\hskip\subfiglabelskip}
\makeatother

\usepackage{hyperref}
\hypersetup{linktocpage=true,colorlinks=true,linkcolor=blue,citecolor=blue,pdfstartview={XYZ 1000 1000 1}}

\usepackage{floatflt,graphicx}

\makeatletter
\newcommand{\doublewedge}{\big@doubleop{\wedge}}
\newcommand{\big@doubleop}[1]{%
  \DOTSB\mathop{\mathpalette\big@doubleop@aux{#1}}\slimits@
}

\newcommand\big@doubleop@aux[2]{%
  \sbox\z@{$\m@th#1#2$}%
  \makebox[1.35\wd\z@][s]{$\m@th#1#2\hss#2$}%
}

\newcommand{\norm}[1]{\left\|#1\right\|}  

\newcommand{\fil}{\mbox{fil}}
\newcommand{\Nrv}{\mbox{Nrv}}

\newcommand{\sk}{\mbox{sk}}

\newcommand{\skCycNrv}{\mbox{sk}_{\mbox{\tiny cyclic}}\mbox{Nrv}}
\newcommand{\skCyc}{\mbox{sk}_{\mbox{\tiny cyclic}}}

\newtheorem{example}{Example}

\newtheorem{lemma}{Lemma}
\newtheorem{theorem}{Theorem}

\newtheorem{observation}{Observation}

\usepackage{xcolor}
\definecolor{light}{gray}{0.80}

\begin{document}

\title[Ghrist Barcoded Video Frames]{Ghrist Barcoded Video Frames.\\ Application in Detecting Persistent Visual Scene\\ 
Surface Shapes captured in Videos}

\author[Arjuna P.H. Don]{Arjuna P.H. Don}
\address{
	Computational Intelligence Laboratory,
	University of Manitoba, WPG, MB, R3T 5V6, Canada}
\email{pilippua@myumanitoba.ca}
	
\author[James F. Peters]{James F. Peters}
\address{
Computational Intelligence Laboratory,
University of Manitoba, WPG, MB, R3T 5V6, Canada and
Department of Mathematics, Faculty of Arts and Sciences, Ad\.{i}yaman University, 02040 
Ad\.{i}yaman, Turkey}
\thanks{The research has been supported by the Natural Sciences \&
Engineering Research Council of Canada (NSERC) discovery grant 185986, Instituto Nazionale di Alta Matematica (INdAM) Francesco Severi, Gruppo Nazionale  per le Strutture Algebriche, Geometriche e Loro Applicazioni grant 9 920160 000362, n.prot U 2016/000036 and Scientific and Technological Research Council of Turkey (T\"{U}B\.{I}TAK) Scientific Human
Resources Development (BIDEB) under grant no: 2221-1059B211301223.}
\email{James.Peters3@umanitoba.ca}

\dedicatory{Dedicated to  Enrico Betti and Paul Alexandroff}


\date{}

\begin{abstract}
This article introduces an application of Ghrist barcodes in the study of persistent Betti numbers derived from vortex nerve complexes found in triangulations of video frames.   A Ghrist barcode is a topology of data pictograph useful in representing the persistence of the features of changing shapes. The basic approach is to introduce a free Abelian group representation of intersecting filled polygons on the barycenters of the triangles of Alexandroff nerves.   An Alexandroff nerve is a maximal collection of triangles with a common vertex in the triangulation of a finite, bounded planar region.   In our case, the planar region is a video frame.   A Betti number is a count of the number of generators in a finite Abelian group.   The focus here is on the persistent Betti numbers across sequences of triangulated video frames.   Each Betti number is mapped to an entry in a Ghrist barcode.   Two main results are given, namely, vortex nerves are Edelsbrunner-Harer nerve complexes and the Betti number of a vortex nerve equals $k+2$ for a vortex nerve containing $k$ edges attached between a pair of vortex cycles in the nerve. 
\end{abstract}

\maketitle

\section{Introduction}
This paper introduces an application of Ghrist barcodes in classifying and identifying persistent video frame shapes.   This approach to video frame shape detection and analysis provides a foundation for machine learning in the study of persistent video frame shapes and a vidoe-shrinking approach to solving the big data problem relative to videos containing thousands of redundant frames. 

A \emph{\bf Ghrist barcode}, usually called a \emph{\bf persistence barcode}, is a topology-of-data pictograph that represents that appearance and disappearance of consecutive sequences of video frames having a particular feature value~\cite{Ghrist2008BAMSbarcodePersistence},~\cite[\S 5.13, pp. 104-106]{Ghrist2014elementaryAppliedGeometry}.   The origin of topology-of-data barcodes can be traced back to H. Edelsbrunner, D. Letscher and A. Zomorodian~\cite{EdelsbrunnerLetscherZomorodian2000IEEEbarcode},~\cite{EdelsbrunnerLetscherZomorodian2001DCGbarcode}.    For a complete view of the landscape for a topology-of-data barcode viewed as a multiset of intervals\footnote{Many thanks to Vidit Nanda for pointing this out.}, see J.A. Perea~\cite{Perea2018arXivhomologyBarcode}. In this paper, the Betti number of a triangulated video frame vortex nerve is represented by consecutive sequences of frames with holes (gaps) between them.    A \emph{vortex nerve} is a collection of nesting, possibly overlapping filled vortexes attached to each other and have nonempty intersection~\cite{AhmadPeters2018centroidalVortices,PetersRamanna2018shapeDescriptions,Peters2018JMSMvortexNerves,Peters2018AlMSproximalPlanarShapes,Peters2017AMSJshapeSignature}.  A \emph{filled vortex} has a boundary that is a simple closed curve and a nonempty interior.   

Each of the vertices on the edges of a vortex nerve is a barycenter (intersection of triangle median lines) in a triangle in an Alexandroff nerve, introduced by P. Alexandroff [Aleksandrov]~\cite[\S 31, p. 39]{Alexandroff1932elementaryConcepts},~\cite{Alexandroff1926MAnnNerfTheorem} and elaborated in~\cite[Vol. 3, p. 67]{Alexandroff1956combinatorialTopology},~\cite[\S 2.11, pp. 160-161]{AlexandroffHopf1935Topologie}.   Briefly, an \emph{Alexandroff nerve} is a collection of triangles with a common vertex.   Each vortex in a nerve has a corresponding cyclic group with a generating element and a Betti number equal to 1.  The \emph{Betti number} of a vortex nerve is a count of the number of generating elements in the nerve.    A sample Ghrist barcode that records the multiple occurrences of video frames containing a vortex nerve with Betti number equal to 8 is shown in Fig.~\ref{fig:bett2}.    Each row of the video Ghrist barcode in Fig.~\ref{fig:bett2} represents a sequence of frames containing vortex nerves with the same Betti number. 

Vortex nerves in triangulated video frames are important in this work, since the edges and interior each vortex nerve reveal paths of reflected light between video frame dark regions (image holes).   Because of their very simple structure, it is a straightforward task to compute the Betti number of a vortex nerve and an equally straightforward task to derive a Betti-number based barcode for each triangulated video.   For a particular Betti number, it is common to find sequences of video frames separated by gaps (frames with a different Betti number) in an entire video.    

This observation leads to the production of focused (reduced) videos containing only video frames with a particular video number.
Here, the basic approach, inspired by B. Le, H. Nguyen and D. Tran~\cite{LeNguyenTran2014videoFrameWatermarking},  
is to shrink a video so that only frames with a recurrent vortex nerve Betti number of interest are retained in a video.   Video shrinking ({\em i.e.}, video frame elimination) is achieved by deleting those frames represented by gaps between occurrences of frames that yield consecutive sequences of a particular vortex nerve Betti number.   

The end result of this approach is easy as well quick access to triangulated video frames that represent a particular Betti number.   The beneficial side-effect of reduced videos is the focus on the minute changes in persistent surface shapes recorded in sequences of frames with a single Betti number orientation.
%
\begin{figure}[!ht]
	\centering	
%
\psscalebox{0.45 0.45} 
{
	\begin{pspicture}(0,-7.2380953)(27.828571,7.2380953)
	\rput(18,-3.5){\includegraphics[width=120mm]{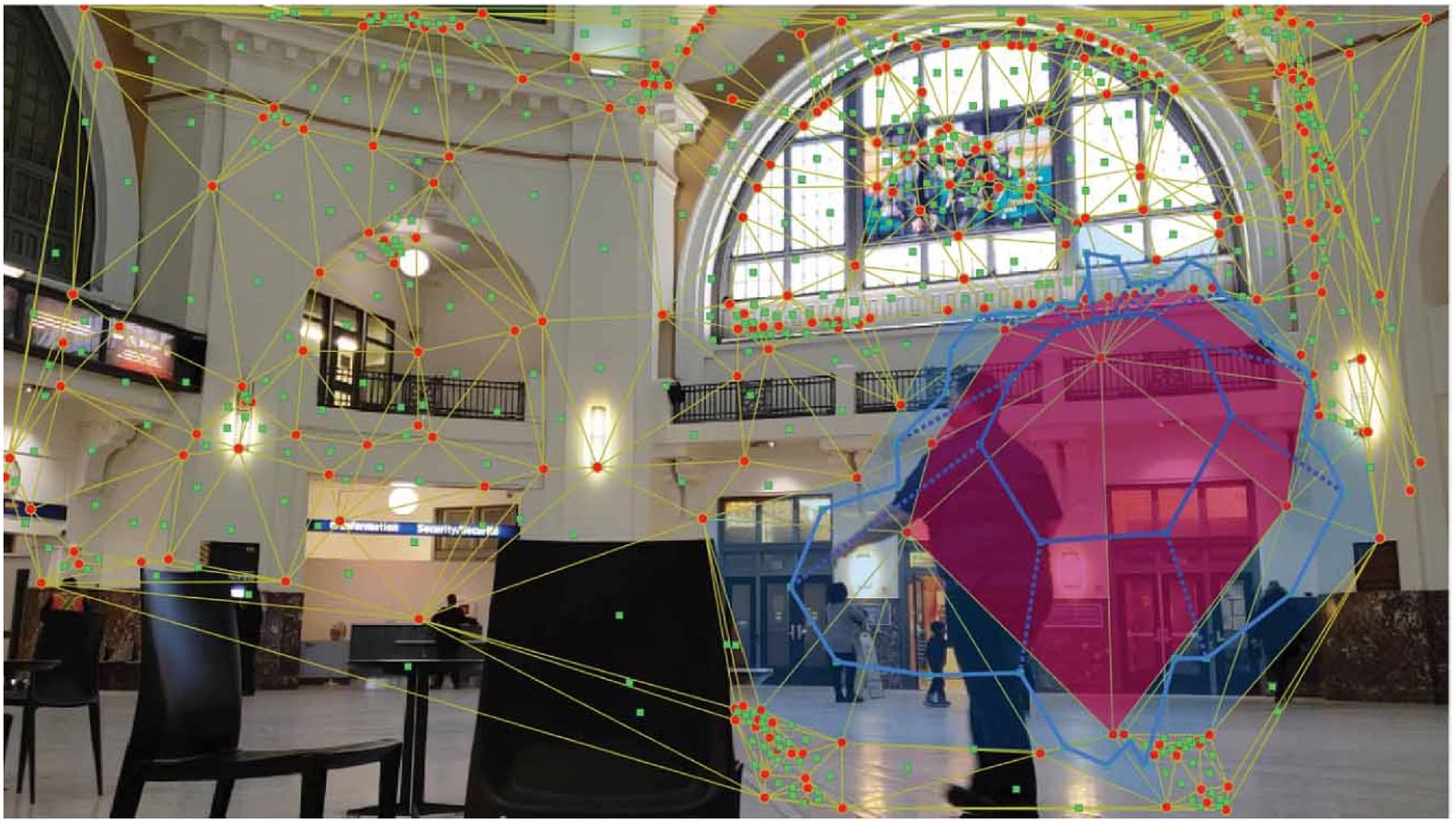}}
	\rput(13,4.2){\includegraphics[width=250mm]{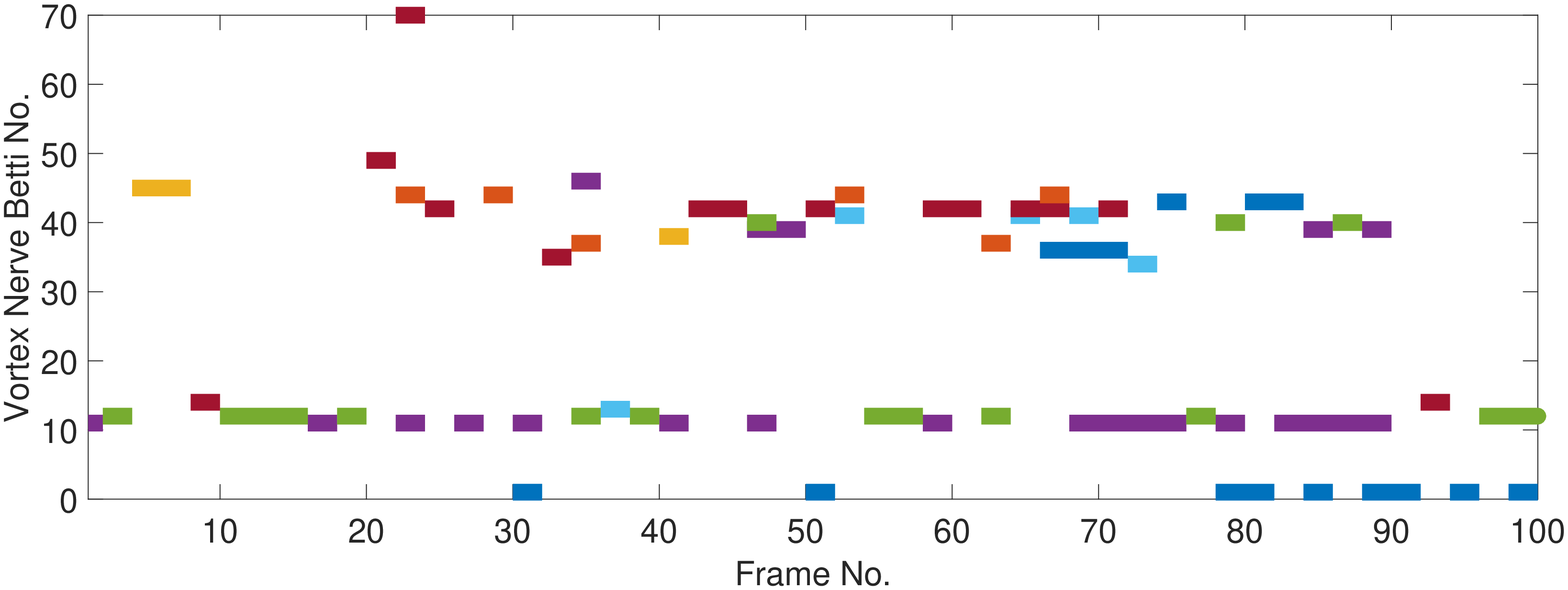}}
	\pspolygon[linecolor=black, 
	linewidth=0.04](4.120086,-2.2731867)(4.120086,-4.0509644)(5.720086,-4.7620754)(7.320086,-4.0509644)(7.320086,-2.2731867)(5.720086,-1.5620756)(5.720086,-1.5620756)(5.720086,-1.5620756)
	\pspolygon[linecolor=black, 
	linewidth=0.04](2.6197853,-1.3809191)(3.8733225,0.0040802136)(5.7051888,0.3601787)(7.3776126,-0.08394513)(8.850173,-1.3180654)(9.258842,-3.1025758)(8.818921,-4.7564554)(7.5653834,-6.1414547)(5.7185273,-6.556803)(4.081527,-6.142655)(2.6189597,-4.869035)(2.2252798,-3.0252748)(2.215063,-2.9806619)
	\psline[linecolor=black, linewidth=0.04](7.3073173,-2.2666667)(8.848781,-1.300813)
	\psline[linecolor=black, linewidth=0.04](7.3170733,-4.022764)(8.809756,-4.744715)
	\psline[linecolor=black, linewidth=0.04](5.717073,-4.744715)(5.736585,-6.5398374)
	\psline[linecolor=black, linewidth=0.04](4.107317,-4.0325203)(2.6243901,-4.8422766)
	\psline[linecolor=black, linewidth=0.04](4.126829,-2.2569106)(2.604878,-1.3788618)
	\psline[linecolor=black, linewidth=0.04](5.717073,-1.5447154)(5.707317,0.37723577)
	\psdots[linecolor=black, fillstyle=solid,fillcolor=green, dotstyle=o, dotsize=0.14, 
	fillcolor=green](4.129976,-2.2736974)
	\rput[bl](4.2002754,-2.6077535){5a}
	\rput[bl](1.9613674,-1.3288167){10b}
	\rput[bl](1.625636,-3.1517243){9b}
	\rput[bl](6.705263,-3.9461427){2a}
	\psdots[linecolor=black, fillstyle=solid,fillcolor=blue, dotstyle=o, dotsize=0.14, 
	fillcolor=blue](8.833881,-1.3191009)
	\psdots[linecolor=black, fillstyle=solid,fillcolor=green, dotstyle=o, dotsize=0.14, 
	fillcolor=green](7.3183813,-4.036016)
	\psdots[linecolor=black, fillstyle=solid,fillcolor=green, dotstyle=o, dotsize=0.14, 
	fillcolor=green](5.72128,-4.751958)
	\psdots[linecolor=black, fillstyle=solid,fillcolor=green, dotstyle=o, dotsize=0.14, 
	fillcolor=green](4.1183815,-4.036016)
	\psdots[linecolor=black, fillstyle=solid,fillcolor=green, dotstyle=o, dotsize=0.14, 
	fillcolor=green](7.3096857,-2.2765958)
	\psdots[linecolor=black, fillstyle=solid,fillcolor=green, dotstyle=o, dotsize=0.14, 
	fillcolor=green](5.7125845,-1.5780451)
	\psdots[linecolor=black, fillstyle=solid,fillcolor=blue, dotstyle=o, dotsize=0.14, 
	fillcolor=blue](5.7106237,0.35528824)
	\psdots[linecolor=black, fillstyle=solid,fillcolor=blue, dotstyle=o, dotsize=0.14, 
	fillcolor=blue](3.8929765,0.008229426)
	\psdots[linecolor=black, fillstyle=solid,fillcolor=blue, dotstyle=o, dotsize=0.14, 
	fillcolor=blue](9.257683,-3.0917706)
	\psdots[linecolor=black, fillstyle=solid,fillcolor=blue, dotstyle=o, dotsize=0.14, 
	fillcolor=blue](8.804741,-4.744712)
	\psdots[linecolor=black, fillstyle=solid,fillcolor=blue, dotstyle=o, dotsize=0.14, 
	fillcolor=blue](7.5518003,-6.1270647)
	\psdots[linecolor=black, fillstyle=solid,fillcolor=blue, dotstyle=o, dotsize=0.14, 
	fillcolor=blue](5.7341533,-6.5447116)
	\psdots[linecolor=black, fillstyle=solid,fillcolor=blue, dotstyle=o, dotsize=0.14, 
	fillcolor=blue](2.6282706,-4.850594)
	\psdots[linecolor=black, fillstyle=solid,fillcolor=blue, dotstyle=o, dotsize=0.14, 
	fillcolor=blue](4.069447,-6.138829)
	\psdots[linecolor=black, fillstyle=solid,fillcolor=blue, dotstyle=o, dotsize=0.14, 
	fillcolor=blue](7.3659177,-0.0908656)
	\psdots[linecolor=black, fillstyle=solid,fillcolor=blue, dotstyle=o, dotsize=0.14, 
	fillcolor=blue](2.2223883,-2.9858882)
	\psdots[linecolor=black, fillstyle=solid,fillcolor=blue, dotstyle=o, dotsize=0.14, 
	fillcolor=blue](2.6282706,-1.3917706)
	\rput[bl](5.561374,-1.9721452){0a}
	\rput[bl](6.744007,-2.5578594){1a}
	\rput[bl](5.5378447,-4.6427336){3a}
	\rput[bl](4.2143154,-4.054498){4a}
	\rput[bl](5.490786,0.46146828){0b}
	\rput[bl](7.337845,0.045501888){1b}
	\rput[bl](8.937844,-1.319204){2b}
	\rput[bl](9.367256,-3.2074392){3b}
	\rput[bl](2.1966681,-5.1250863){8b}
	\rput[bl](3.7496095,-6.48391){7b}
	\rput[bl](5.579021,-6.901557){6b}
	\rput[bl](7.561374,-6.4486156){5b}
	\rput[bl](8.884904,-4.972145){4b}
	\rput[bl](3.2672565,0.12785482){11b}
	\rput[bl](4.6241193,-0.109960295){$\boldsymbol{\left\langle b\right\rangle}$}
	\rput[bl](7.8907857,-2.1194842){$\boldsymbol{\left\langle e_1\right\rangle}$}
	\rput[bl](5.7479286,-0.8147222){$\boldsymbol{\left\langle e_6\right\rangle}$}
	\rput[bl](3.2241192,-1.8051984){$\boldsymbol{\left\langle e_5\right\rangle}$}
	\rput[bl](2.986024,-4.2813888){$\boldsymbol{\left\langle e_4\right\rangle}$}
	\rput[bl](4.8907857,-5.7385316){$\boldsymbol{\left\langle e_3\right\rangle}$}
	\rput[bl](7.5574527,-4.7385316){$\boldsymbol{\left\langle e_2\right\rangle}$}
	\rput[bl](4.690786,-2.224246){$\boldsymbol{\left\langle a\right\rangle}$}
	\psline[linecolor=red, linewidth=0.04, arrowsize=0.05291667cm 
	2.0,arrowlength=1.4,arrowinset=0.0]{->}(4.857143,0.43809524)(3.8190476,2.1904762)
	\psline[linecolor=red, linewidth=0.04, arrowsize=0.05291667cm 
	2.0,arrowlength=1.4,arrowinset=0.0]{->}(11.866667,-3)(10.095238,-3)
	\end{pspicture}
}
	\caption[]{Video barcode entry derived from a triangulated frame with vortex nerve 
	Betti no. = 8}
	\label{fig:bett2}
\end{figure}

\section{Preliminaries}
This section briefly introduces the methods used to derive vortex nerves and their Betti numbers as well as the method used to construct the Ghrist barcode for a video.   The sample barcode in Fig.~\ref{fig:bett2} is the result of the following steps leading from the detection of a vortex nerve in each triangulated video frame to the Betti number representing the video frame, {\em i.e.},\\
$\mbox{}$\\
\fbox{\bf Steps leading to a Ghrist barcode entry for video frame Betti number:}\\
\begin{compactenum}[{\bf Step} 1]
\item {\bf Select a video $v$}, a sequence of frames $fr[1],\dots,\boldsymbol{fr[i]},\dots,fr[n]$. 
\item {\bf Select $i^{th}$ frame} $fr[i ]$ in video $v$.
\item {\bf Frame $fr[i ]$ frame} maps to a set of centroids $S$ on image holes (black regions in
a binary image).
\item {\bf Frame Centroids $S$} map to a set of nonoverlapping triangles $\left\{\bigtriangleup\right\}$.
\item {\bf Triangles $\left\{\bigtriangleup\right\}$} map to a set of barycenters $B$.
\item {\bf Barycenters $B$} map to a vortex nerve \colorbox{green}{$\skCycNrv E$}.
\item {\bf Vortex nerve $\skCycNrv E$} maps to Betti number $\mathcal{B}(\skCycNrv E)$, which is a count of the number
of overlapping vortexes plus the number of edges (called cusp filaments) connected between the vortexes.  
\item {\bf Betti number $\boldsymbol{\mathcal{B}(\skCycNrv E)}$} maps to an entry (tiny bar) in a row of a video Ghrist barcode.
\end{compactenum}

\begin{figure}[!ht]
\centering	
\begin{pspicture}
(-1.0,1.0)(8.0,1.5)
 \rput(-1.2, 1.0){$fr [i ]$ }
 \psTextFrame[](0.0,0.8)(0.8,1.3){$S$}
 \psline[linecolor=black, linewidth=0.04, arrowsize=0.05291667cm 
			2.0,arrowlength=1.4,arrowinset=0.0]{->}(-0.8,1.0)(-0.2,1.0)
 \psTextFrame[](1.5,0.8)(2.2,1.3){$\left\{\bigtriangleup\right\}$}
 \psline[linecolor=black, linewidth=0.04, arrowsize=0.05291667cm 
			2.0,arrowlength=1.4,arrowinset=0.0]{->}(0.9,1.0)(1.45,1.0)
 \psTextFrame[](3.0,0.8)(3.8,1.3){$B$}
 \psline[linecolor=black, linewidth=0.04, arrowsize=0.05291667cm 
			2.0,arrowlength=1.4,arrowinset=0.0]{->}(2.3,1.0)(2.85,1.0)
 \psTextFrame*[linecolor=green!20](4.5,0.8)(6.2,1.3){\footnotesize $\boldsymbol{\skCycNrv E}$}
 \psline[linecolor=black, linewidth=0.04, arrowsize=0.05291667cm 
			2.0,arrowlength=1.4,arrowinset=0.0]{->}(3.9,1.0)(4.45,1.0)
 \psframe[](4.5,0.8)(6.2,1.3)
 \rput(7.8, 1.0){\footnotesize $\boldsymbol{\mathcal{B}(\skCycNrv E)}$}
 \psline[linecolor=black, linewidth=0.04, arrowsize=0.09291667cm 
			2.1,arrowlength=1.8,arrowinset=0.0]{->}(6.3,1.0)(6.75,1.0)
\end{pspicture}
\caption[]{Stages from video frame $fr [i ]$ to Betti no. $\mathcal{B}(\skCycNrv E)$ of vortex nerve $\skCycNrv E$}
\label{fig:nerveBettiNumber}
\end{figure}

The mappings in steps 3 to 7 are shown in Fig.~\ref{fig:nerveBettiNumber}.    Each mapping in these steps establishes
a correspondence between an object such as the set of centroids $S$ 

\begin{example}{\bf Mapping of a Frame Betti number to a Video Ghrist Barcode Entry}.\\
An illustration of the steps leading to a video Ghrist barcode entry is partially represented in Fig.~\ref{fig:bett2}.   In this case, a vortex nerve with Betti number equal to 8 is derived from a sample video frame.   In Fig.~\ref{fig:bett2},
the barycenters of the triangles on the sample video frame are mapped to a vortex nerve (step 6) \colorbox{green}{$\skCycNrv E$}.  The Betti number of the nerve $\skCycNrv E$ is mapped to a tiny bar entry in a row of the Ghrist barcode (step 8).
\qquad \textcolor{blue}{\Squaresteel}
\end{example}

\subsection{Holes, centroids, triangulation, barycenters and vortexes}$\mbox{}$\\
A \emph{hole} is a surface retion that absorbs light.   In terms of optical sensors that record the intensity and colour of light reflected from visual scene surfaces, near zero intensities correspond to dark (hole) surface regions.   In the morphology of binary images, each centroid is the center of mass of a dark image region that corresponds to a visual scene surface hole.   Let $X$ be a set of points in a $n\times m$ rectangular 2D region containing points with coordinates $\left(x_i,y_i\right),i=1,\dots,n$ in the Euclidean plane.   Then, for example, the coordinates $x_c,y_c$ of a centroid in an $n\times m$  2D rectangular are
\[
x_c = \frac{1}{n}\mathop{\sum}\limits_{i=1}^{n} x_i, y_c = \frac{1}{m}\mathop{\sum}\limits_{i=1}^{m} y_i.
\]

\begin{algorithm}[!ht]
\caption{\bf  Filled Delaunay Triangle Construction}\label{alg:DelaunayTriangleConstruction}

\SetKwData{Left}{left}
\SetKwData{This}{this}
\SetKwData{Up}{up}
\SetKwFunction{Union}{Union}
\SetKwFunction{FindCompress}{FindCompress}
\SetKwInOut{Input}{Input}
\SetKwInOut{Output}{Output}
\SetKwComment{tcc}{/*}{*/}

\Input{Set of centroids $S$ on the holes (dark regions) on video frame}
\Output{Delaunay Triangle Construction}
\emph{Let $p$ be a centroid in $S$}\;
\emph{{\bf Triangle Vertexes Selection Step}; Select centroids $q,r\in S$ nearest $p\in S$}\;
\emph{Draw edge $\arc{pq}$ on a closed half plane $\pi_{pq}$ that covers $r\in S$}\;
\emph{Draw edge $\arc{pr}$ on a closed half plane $\pi_{pr}$ that covers $q\in S$}\;
\emph{Draw edge $\arc{qr}$ on a closed half plane $\pi_{qr}$ that covers $p\in S$}\;
\emph{Edges on triangle $\bigtriangleup(pqr)$ are on intersecting half planes covering $\bigtriangleup(pqr)$}\;
/* $\bigtriangleup(pqr)$ is a filled Delaunay triangle */ \;
\end{algorithm}

A centroid is also called a seed point, which is used as a vertex in the triangulation of a digital image.   By selecting a set of seed points that are image centroids, it is then possible to construct a collection of what are known as Delaunay triangles on the image.   Let $p,q,r$ be three neighbouring image centroids is a set of seed points $S$ and let $\pi_{pq},\pi_{pr},\pi_{qr}$ be three half planes.   Then Alg.~\ref{alg:DelaunayTriangleConstruction} is used to construct a filled Delaunay triangle.

\begin{figure}[!ht]
\centering
\begin{pspicture}
(-0.5,-1.2)(4,0)
\psline*[linecolor=gray!20](0,-1)(2,0)(4,-1) 
\psline[linewidth=0.5pt](0,-1)(2,0) 
\psline[linewidth=0.5pt](2,0)(4,-1) 
\psline[linewidth=0.5pt](4,-1)(0,-1) 
\psline[linecolor=blue, linewidth=0.5pt](4,-1)(1,-0.5) 
\psline[linecolor=blue, linewidth=0.5pt](2,0)(2,-1) 
\psline[linecolor=blue, linewidth=0.5pt](0,-1)(3,-0.5) 
\psdots[dotstyle=o, linewidth=0.5pt,linecolor = black, fillcolor = red]
(0,-1)(2,0)(4,-1)
\psdots[dotstyle=o, linewidth=0.5pt,linecolor = black, fillcolor = yellow]
(1,-0.5)(2,-1)(3,-0.5) 
\psdots[dotstyle=o, linewidth=0.5pt,linecolor = black, fillcolor = green]
(2,-0.67) 
\rput(-0.2,-1){$\boldsymbol{p}$}\rput(4.2,-1){$\boldsymbol{q}$}
\rput(2.0,0.2){$\boldsymbol{r}$}
\rput(2.2,-0.4){$\boldsymbol{b}$}
\end{pspicture}
\caption[]{Filled Delaunay triangle $\boldsymbol{\bigtriangleup(pqr)}$ barycenter $\boldsymbol{b}$}
\label{fig:barycenter}
\end{figure}

The intersection of the median lines of a filled Delaunay triangle is \emph{barycenter} of the triangle.

\begin{example}{\bf Sample Barycenters of Filled Dalaunay Triangles}.\\
Let $\boldsymbol{\bigtriangleup(pqr)}$ be a filled Delaunay triangle as shown in  Fig.~\ref{fig:barycenter}.
The \emph{median line} of a triangle is the line drawn from a triangle vertex to the midpoint of the side opposite the vertex.   Three median lines are also shown in Fig.~\ref{fig:barycenter}.   The barycenter of $\boldsymbol{\bigtriangleup(pqr)}$ is located at the intersection of the median lines (shown with a green \textcolor{green}{$\boldsymbol{\bullet}$} in Fig.~\ref{fig:barycenter}).   Many other examples of Delaunay triangle barycenters are shown in Fig.~\ref{fig:vortexNerveFromArjuna}.
\qquad \textcolor{blue}{\Squaresteel}
\end{example}

The important thing to notice here is that an image barycenter on a centroidal triangle is in an image region between the dark regions (holes) in a visual scene.    In other words, the barycenter of a centroidal triangle like  $\boldsymbol{\bigtriangleup(pqr)}$ in Fig.~\ref{fig:barycenter} originates from a surface shape that reflects or refracts light bombarding the surface. 

\begin{figure}[!ht]
\centering
\begin{pspicture}
(-0.5,-1.2)(4.5,1.03)
\psline*[linecolor=gray!20](0,-1)(2,0)(4,-1) 
\psline[linewidth=0.5pt](0,-1)(2,0) 
\psline[linewidth=0.5pt](2,0)(4,-1) 
\psline[linewidth=0.5pt](4,-1)(0,-1) 
\psline*[linecolor=gray!20](4,0)(4,1)(2,0) 
\psline*[linecolor=gray!20](4,-1)(4,0)(2,0) 
\psline*[linecolor=gray!20](2,0)(4,1)(0,1) 
\psline*[linecolor=gray!20](2,0)(0,1)(0,0) 
\psline*[linecolor=gray!20](2,0)(0,0)(0,-1) 
\psline[linewidth=0.5pt](2,0)(0,-1)\psline[linewidth=0.5pt](0,0)(0,-1) 
\psline[linewidth=0.5pt](0,1)(4,1)\psline[linewidth=0.5pt](2,0)(0,1) 
\psline[linewidth=0.5pt](0,0)(0,1)\psline[linewidth=0.5pt](0,0)(2,0) 
\psline[linewidth=0.5pt](2,0)(4,0) 
\psline[linewidth=0.5pt](2,0)(4,-1) 
\psline[linewidth=0.5pt](4,-1)(4,0) 
\psline[linewidth=0.5pt,linecolor=blue](2,-0.67)(3.35,-0.30) 
\psline[linewidth=0.5pt,linecolor=blue](3.35,-0.30)(3.35,0.30) 
\psline[linewidth=0.5pt,linecolor=blue](3.35,0.30)(2,0.67) 
\psline[linewidth=0.5pt,linecolor=blue](2,0.67)(0.65,0.30) 
\psline[linewidth=0.5pt,linecolor=blue](0.65,0.30)(0.65,-0.30) 
\psline[linewidth=0.5pt,linecolor=blue](0.65,-0.30)(2,-0.67) 
\rput(2.0,0.87){\footnotesize $\boldsymbol{\skCyc E}$}
\rput(2.0,0.47){\footnotesize $\boldsymbol{\left\langle a\right\rangle}$}
\psline[linewidth=0.5pt](4,0)(4,1)\psline[linewidth=0.5pt](4,1)(2,0) 
\psdots[dotstyle=o, linewidth=0.5pt,linecolor = black, fillcolor = green]
(2,0.67) 
\psdots[dotstyle=o, linewidth=0.5pt,linecolor = black, fillcolor = green]
(3.35,-0.30) 
\psdots[dotstyle=o, linewidth=0.5pt,linecolor = black, fillcolor = green]
(3.35,0.30) 
\psdots[dotstyle=o, linewidth=0.5pt,linecolor = black, fillcolor = green]
(0.65,0.30) 
\psdots[dotstyle=o, linewidth=0.5pt,linecolor = black, fillcolor = green]
(0.65,-0.30) 
\psdots[dotstyle=o, linewidth=0.5pt,linecolor = black, fillcolor = red]
(0,-1)(2,0)(4,-1)(0,0)(0,1)(4,0)(4,1)
\psdots[dotstyle=o, linewidth=0.5pt,linecolor = black, fillcolor = green]
(2,-0.67) 
\end{pspicture}
\caption[]{Barycentric vortex cycle $\boldsymbol{\skCyc E}$ with generator $\boldsymbol{\left\langle a\right\rangle}$ on an Alexandroff nerve}
\label{fig:vortex}
\end{figure}$\mbox{}$\\

\noindent \fbox{\bf Alexandroff Nerve Triangles\ $\boldsymbol{\longmapsto}$\ Barycentric Vortex Cycle:}\\
The next step is to select an Alexandroff nerve $\Nrv E$ ({\em  i.e.}, collection of triangles with a common vertex).   Next, locate the barycenters of the triangles in $\Nrv E$.   Then draw edges of half planes between neighbouring barycenters on nerve $\Nrv E$.   Each barycentric half plane covers the nucleus (common vertex) of  $\Nrv E$.   The end result is a filled \emph{vortex cycle} $E$ (denoted by $\skCyc E$) that captures the reflected light from visual scene surface shapes.   Every vortex cycle has a generator $a$ (denoted by $\left\langle a\right\rangle$), where $a$ is the vertex of an edge in the path between vertex $a$ and any other vertex in the cycle.   A simplifying assumption made here is that every sequence of edges is a multiple of the edge that begins with vertex $a$.    For more about generators, see Section~\ref{sec:generator}.   

\begin{example}{\bf Sample Barycentric Vortex Cycle on an Alexandroff Nerve}.\\
Let $\boldsymbol{\skCyc E}$ be a vortex cycle drawn on the barycenters of a filled Delaunay triangle as shown in  Fig.~\ref{fig:vortex}.
Many other examples of barycentric vortex cycles are shown in Fig.~\ref{fig:vortexNerveFromArjuna}.
\qquad \textcolor{blue}{\Squaresteel}
\end{example}

H. Edelsbrunner and J.L Harer~\cite[\S III.2, p. 59]{Edelsbrunner1999} that, in general, a \emph{nerve} complex is a collection of sets that have nonempty intersection.\\
\vspace{3mm}

\noindent \fbox{\bf Construction of a Barycentric Vortex Nerve:}\\
For the triangles bordering an Alexandroff nerve $\Nrv E$, construct a filled vortex cycle $\skCyc E'$ on the barycenters of the triangles.   Each edge of $\skCyc E'$ is on a half plane that covers the common vertex of the Alexandroff nerve $\Nrv E$.

\begin{lemma}\label{lemma:nestingVortexes}
A collection of nesting, overlapping filled vortexes is an Edelsbrunner-Harer nerve complex.
\end{lemma}
\begin{proof}
Let $\skCyc E, \skCyc E'$ be filled vortex cycles on Alexandroff nerve $\Nrv E$.   $\skCyc E, \skCyc E'$ are nesting vortexes, since vortex $\skCyc E$ is in the interior of $\skCyc E'$.   We have
\[
\skCycNrv E = \left\{\skCyc E, \skCyc E': \skCyc E \cap \skCyc E' \neq \emptyset\right\},
\]
since each vortex cycle is filled and the portion of the plane occupied by $\skCyc E'$ includes $\skCyc E$.
Hence, $\skCycNrv E$ is an Edelsbrunner-Harer nerve complex.
\end{proof}

For each video frame, it is possible to extend outward from a barycentric Alexandroff vortex cycle to form multiple nesting filled vortex cycles, one inside the other.   In that case, the collection of vortex cycles form a large vortex nerve.

\begin{theorem}
A collection of nesting, overlapping filled vortex cycles on a triangulated video frame is an Edelsbrunner-Harer nerve complex.
\end{theorem}
\begin{proof}
Let $\skCycNrv E$ be a collection of nesting, overlapping filled vortex cycles on a video frame.   From Lemma~\ref{lemma:nestingVortexes}, we have
\[
\skCycNrv E = \left\{\skCyc E\in \skCycNrv E: \mathop{\bigcap} \skCyc E \neq \emptyset \neq \emptyset\right\}.
\]
Hence, $\skCycNrv E$ is a nerve complex.
\end{proof}

\begin{figure}
	\centering
	\subfigure[Frame 4-vortex nerve]{
	\includegraphics[width=0.8\linewidth]{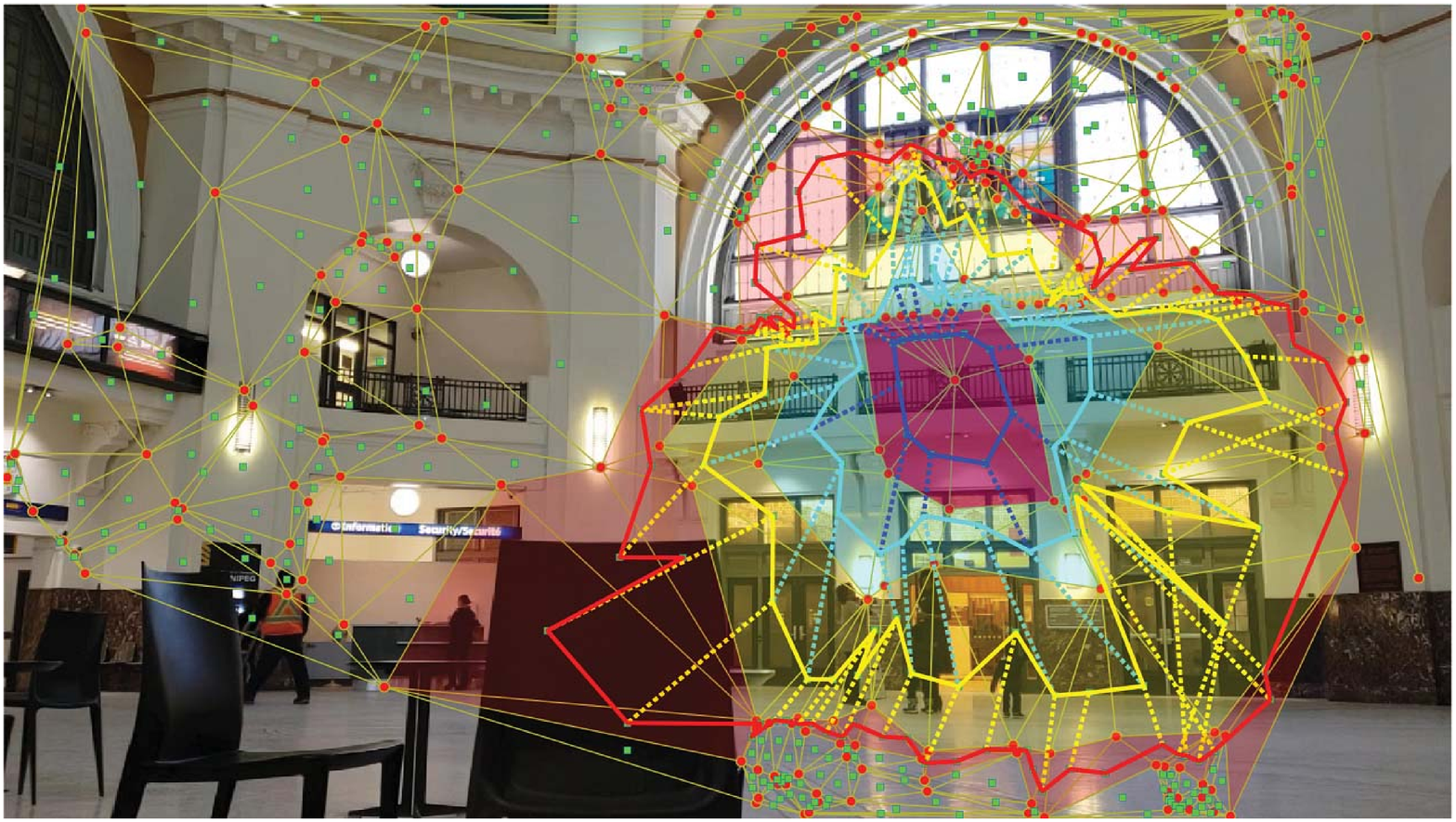}\label{fig:Dsample1_Moment1}}
	\subfigure[Frame 2-vortex nerve]{
	\includegraphics[width=0.8\linewidth]{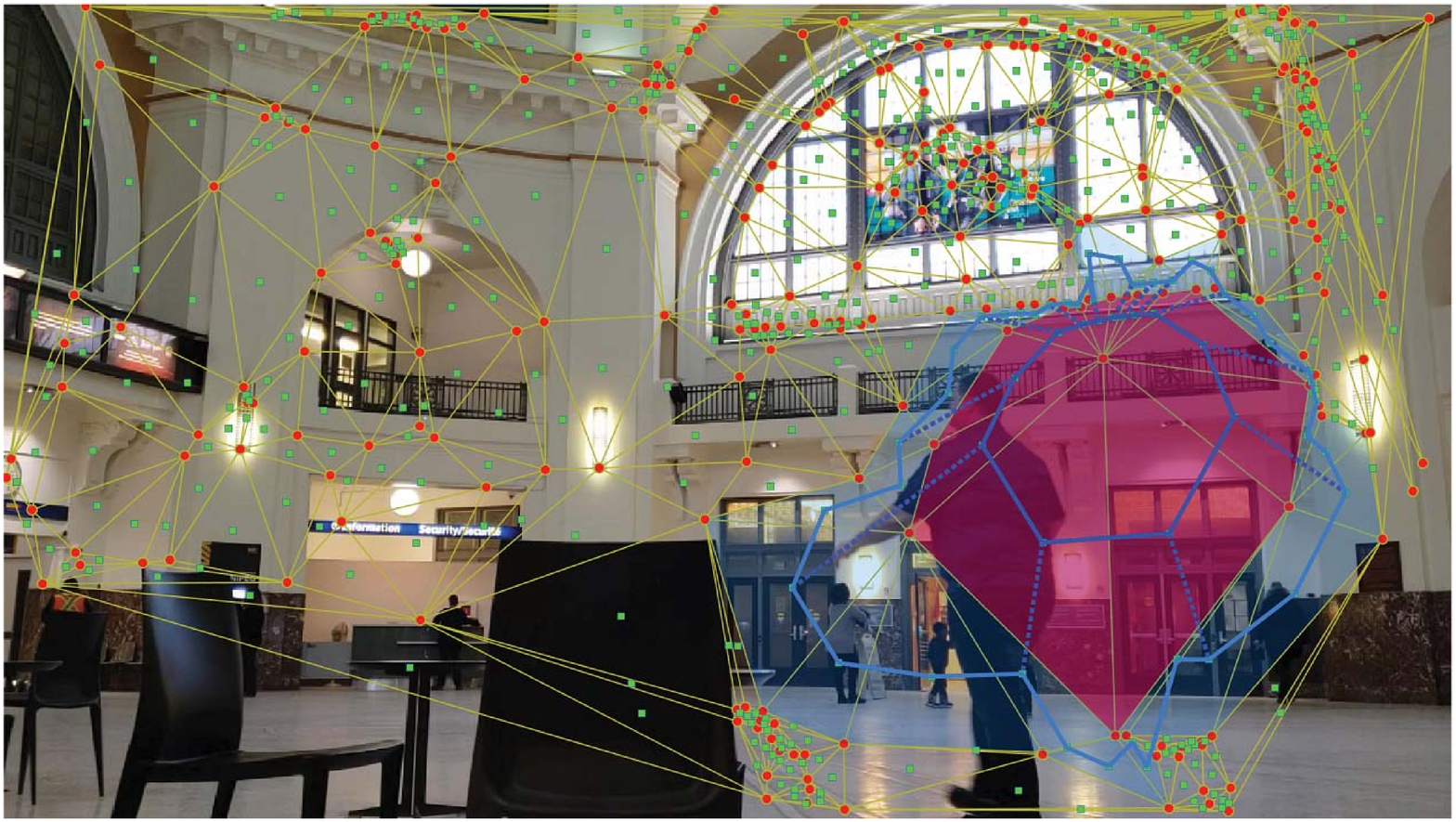}\label{fig:Dsample1_Moment2}}
	 
	\subfigure[Frame dual vortex nerves]{
	\includegraphics[width=0.8\linewidth]{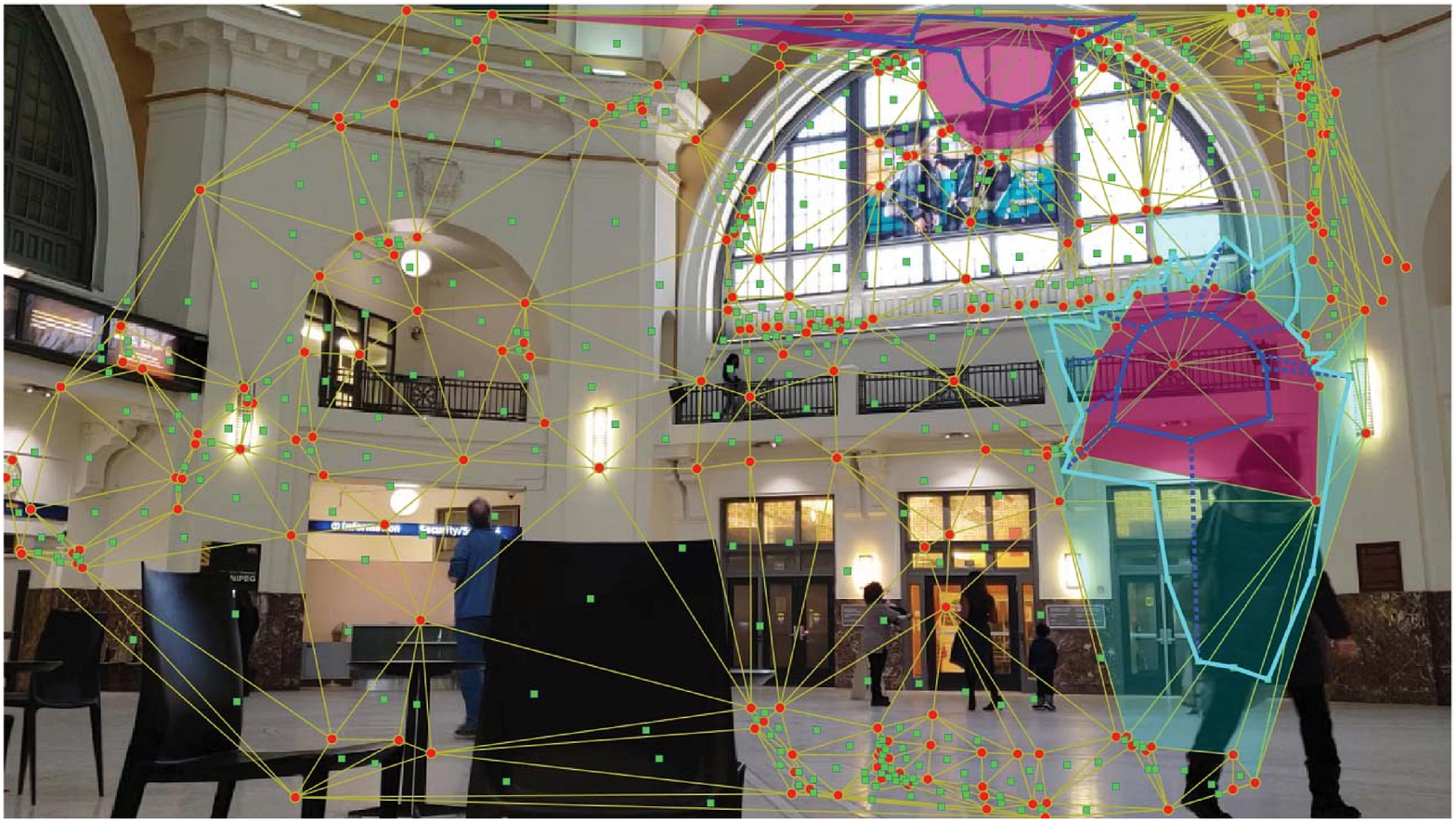}\label{fig:Dsample1_Moment3}}
	 
	\caption{Sample triangulated frames for video 1.}
	\label{fig:Dsample1_Moment}
\end{figure}  

\begin{example}{\bf Sample Video Fame Vortex Nerve Complexes}.\\
A trio of video frame vortex nerve complexes are shown in Fig.~\ref{fig:Dsample1_Moment}.   In Fig.~\ref{fig:Dsample1_Moment1}, for example, a collection of 4 nesting, overlapping vortexes are shown with
\begin{compactenum}[{\bf vortex}.1]
\item (innermost \colorbox{blue!50}{\bf vortex}) with blue edges with 12 vertices.   There is a cusp filament attached between the 12 vertices on the blue
vortex and
\item (\colorbox{blue!20}{\bf vortex}) with more than 20 vertices.   Again, there is a cusp filament attached between more than 20 vertices on the blue
vortex and
\item (\colorbox{yellow}{\bf vortex}) with more than 30 vertices.   Again, there is a cusp filament attached between more than 30 vertices on the blue
vortex and
\item (\colorbox{red!50}{\bf vortex}) with more than 30 vertices.
\end{compactenum}
The vortex nerve in Fig.~\ref{fig:Dsample1_Moment2} consists of a pair of overlapping vortexes with 12 cusp filments attached between the vortexes.   A pair of vortex nerves on the same triangulated video frame are shown in Fig.~\ref{fig:Dsample1_Moment3}.
\qquad \textcolor{blue}{\Squaresteel}
\end{example}

\subsection{Cyclic groups representing video frame vortex nerves.}\label{sec:generator}$\mbox{}$\\
Let $\sk E$ be a filled vortex ({\em i.e.}, also called a skeleton) on a triangulated video frame.
The vertices in vortex $\sk E$ are path-connected.   This means that there is path between every pair of vertices in $\sk E$.   In addition, each $\sk E$ is bi-directional in a vortex nerve.   So, for example, if vertices $p,q$ are on $\sk E$, a movement (traversal) from $p$ to $q$ is represented by $p + q$ and a reverse traversal is represented by $-q + p$.  

The $+$ between path edges reads \textcolor{blue}{\bf attach to}.   No movement is represented by $p + (-p) = 0p$.   In effect, every member $p$ in $\sk E$ has an inverse $-p$ and $0p$ represents an identity element in an algebraic group view of $\sk E$.  Notice that $+$ operation is Abelian.   To see this, do addition modulo 2 on the coefficients in a movement from $p$ to $q$, {\em e.g.},
\begin{align*}
p + q &\longmapsto 1 + 1\mbox{mod}2 = 0,\ \mbox{and}\\ 
q + p&\longmapsto 1 + 1\mbox{mod}2 = 0.
\end{align*}
For such a group, we write $(\sk E, +)$ (called a \emph{cyclic Abelian group}).    A \emph{generator} of skeleton $\sk E$ (denoted by $\left\langle a \right\rangle$) of such a group is a minimum length edge with a distinguished, starting vertex $a\in \sk E$.   Let $V$ be an ordered set of $k$ path-connected vertices in vortex $\sk E$, namely,
\begin{align*}
V &= \overbrace{\left\{v_0, v_1, \dots, v_i, \dots, v_{k-1}\right\}}^{\mbox{\textcolor{blue}{\bf $k$ ordered vertices}}}.\\
\left\langle a \right\rangle &= \overbrace{\norm{v_1 - v_0}.}^{\mbox{\textcolor{blue}{\bf generator $\left\langle a \right\rangle$ = minimum edge-length }}}\\
\end{align*}

\begin{example}{\bf Sample Generator of a Cyclic Group}.\\
Let generator $\left\langle a \right\rangle$ represented by vertex $0b$ in vortex $\sk B$ in Fig.~\ref{fig:bett2}, which is the starting vertex in a minimum length edge $\arc{0b1b}$.   Then, for example, vertex $p2\in \sk B$ can written as a multiple of edge $\arc{0b1b}$, {\em i.e.},
\begin{align*}
p2 &= \arc{0b1b} + \arc{1b2b} +\cdots+ \arc{8b9b}\\
     &\overbrace{\longmapsto 1b + 1b  +\cdots+ 1b}^{\mbox{\textcolor{blue}{\bf maps to a multiple of generator $\left\langle b \right\rangle$}}}\\
		 &= 9b. \ \mbox{\qquad \textcolor{blue}{\Squaresteel}}
\end{align*}
\end{example}

\begin{figure}[!ht]
\centering
\psscalebox{0.8 0.8} 
		{
\begin{pspicture}
(0,-3.8845036)(8.979828,3.8845036)
			\definecolor{colour1}{rgb}{1.0,0.0,0.2}
			\pspolygon[linecolor=black, 
			linewidth=0.04](3.142069,0.8229681)(3.142069,-0.9548096)(4.7420692,-1.6659207)(6.342069,-0.9548096)(6.342069,0.8229681)(4.7420692,1.5340793)(4.7420692,1.5340793)(4.7420692,1.5340793)
			\psline[linecolor=black, linewidth=0.04, arrowsize=0.05291667cm 
			2.0,arrowlength=1.4,arrowinset=0.0]{->}(3.2207608,1.1094712)(4.5018067,1.6641552)
			\psline[linecolor=black, linewidth=0.04, arrowsize=0.05291667cm 
			2.0,arrowlength=1.4,arrowinset=0.0]{->}(5.0169344,1.6518395)(6.251114,1.0897926)
			\psline[linecolor=black, linewidth=0.04, arrowsize=0.05291667cm 
			2.0,arrowlength=1.4,arrowinset=0.0]{->}(1.3812699,1.6993971)(0.98778266,0.28341147)
			\psline[linecolor=black, linewidth=0.04, arrowsize=0.05291667cm 
			2.0,arrowlength=1.4,arrowinset=0.0]{->}(7.9219904,2.0973988)(6.6736045,3.1297426)
			\psline[linecolor=black, linewidth=0.04, arrowsize=0.05291667cm 
			2.0,arrowlength=1.4,arrowinset=0.0]{->}(6.5935817,-0.7384424)(6.5908227,0.67950064)
			\pspolygon[linecolor=black, 
			linewidth=0.04](1.6417683,1.7152357)(2.8953054,3.100235)(4.727172,3.4563336)(6.3995953,3.0122097)(7.872156,1.7780894)(8.280826,-0.0064210207)(7.8409038,-1.6603006)(6.5873666,-3.0453)(4.7405105,-3.4606483)(3.1035101,-3.0465002)(1.6409427,-1.7728801)(1.2472628,0.07088011)(1.2370461,0.11549287)
			\psline[linecolor=black, 
			linewidth=0.04](6.3293004,0.82948816)(7.870764,1.7953418)
			\psline[linecolor=black, linewidth=0.04, linestyle=solid, 
			dash=0.17638889cm 0.10583334cm](6.3390565,-0.9266094)(7.8317394,-1.6485606)
			\psline[linecolor=black, linewidth=0.04, linestyle=solid, 
			dash=0.17638889cm 0.10583334cm](4.739056,-1.6485606)(4.7585683,-3.4436827)
			\psline[linecolor=black, linewidth=0.04, linestyle=solid, 
			dash=0.17638889cm 0.10583334cm](3.1293,-0.9363655)(1.6463734,-1.7461216)
			\psline[linecolor=black, linewidth=0.04, linestyle=solid, 
			dash=0.17638889cm 0.10583334cm](3.1488123,0.83924425)(1.6268611,1.717293)
			\psline[linecolor=black, linewidth=0.04, linestyle=solid, 
			dash=0.17638889cm 0.10583334cm](4.739056,1.5514394)(4.7293,3.4733906)
			\psdots[linecolor=black, fillstyle=solid,fillcolor=green, dotstyle=o, 
			dotsize=0.14, fillcolor=green](3.1519587,0.82245755)
			\rput[bl](3.1936872,0.39316317){5a}
			\pscircle[linecolor=colour1, linewidth=0.04, 
			dimen=outer](6.342069,-0.9383741){0.17919655}
			\psframe[linecolor=colour1, linewidth=0.04, 
			dimen=outer](1.8176789,1.8751123)(1.4865456,1.5439788)
			\psline[linecolor=black, linewidth=0.04, arrowsize=0.05291667cm 
			2.0,arrowlength=1.4,arrowinset=0.0]{->}(6.3768587,1.1533736)(7.511563,1.8543992)
			\psline[linecolor=black, linewidth=0.04, arrowsize=0.05291667cm 
			2.0,arrowlength=1.4,arrowinset=0.0]{->}(6.365893,3.2730086)(4.8736043,3.680962)
			\psline[linecolor=black, linewidth=0.04, arrowsize=0.05291667cm 
			2.0,arrowlength=1.4,arrowinset=0.0]{->}(4.4535317,3.6649516)(2.9239855,3.372807)
			\psline[linecolor=black, linewidth=0.04, arrowsize=0.05291667cm 
			2.0,arrowlength=1.4,arrowinset=0.0]{->}(2.5838044,3.1470175)(1.5434978,1.9923191)
			\rput[bl](0.8,1.8673382){10b}
			\rput[bl](0.7,0.011097169){9b}
			\rput[bl](5.979627,-0.8309401){2a}
			\rput[bl](6.987805,1.0){$\boldsymbol{\left\langle e_1\right\rangle}$}
      \rput[bl](6.5,-1.5){$\boldsymbol{\left\langle e_2\right\rangle}$}
			\rput[bl](4.15,-2.5){$\boldsymbol{\left\langle e_3\right\rangle}$}
			\rput[bl](2.25,-1.1){$\boldsymbol{\left\langle e_4\right\rangle}$}
			\rput[bl](2.25, 1.3){$\boldsymbol{\left\langle e_5\right\rangle}$}
			\rput[bl](4.85, 2.0){$\boldsymbol{\left\langle e_6\right\rangle}$}
				\rput[bl](4.55, 0.8){$\boldsymbol{\left\langle a\right\rangle}$}
				\rput[bl](4.25, 2.8){$\boldsymbol{\left\langle b\right\rangle}$}
			\psdots[linecolor=black, dotstyle=o, dotsize=0.14, 
			fillcolor=blue](7.8558645,1.777054)
			\psdots[linecolor=black, dotstyle=o, dotsize=0.14, 
			fillcolor=green](6.3403645,-0.9398613)
			\psdots[linecolor=black, dotstyle=o, dotsize=0.14, 
			fillcolor=green](4.7432632,-1.6558033)
			\psdots[linecolor=black, dotstyle=o, dotsize=0.14, 
			fillcolor=green](3.1403646,-0.9398613)
			\psdots[linecolor=black, dotstyle=o, dotsize=0.14, 
			fillcolor=green](6.331669,0.81955904)
			\psdots[linecolor=black, dotstyle=o, dotsize=0.14, 
			fillcolor=green](4.7345676,1.5181098)
			\psdots[linecolor=black, dotstyle=o, dotsize=0.14, 
			fillcolor=green](4.732607,3.451443)
			\psdots[linecolor=black, dotstyle=o, dotsize=0.14, 
			fillcolor=green](2.9149597,3.1043842)
			\psdots[linecolor=black, dotstyle=o, dotsize=0.14, 
			fillcolor=green](8.279666,0.004384249)
			\psdots[linecolor=black, dotstyle=o, dotsize=0.14, 
			fillcolor=green](7.8267245,-1.648557)
			\psdots[linecolor=black, dotstyle=o, dotsize=0.14, 
			fillcolor=green](6.5737834,-3.0309098)
			\psdots[linecolor=black, dotstyle=o, dotsize=0.14, 
			fillcolor=green](4.756136,-3.448557)
			\psdots[linecolor=black, dotstyle=o, dotsize=0.14, 
			fillcolor=green](1.6502538,-1.7544392)
			\psdots[linecolor=black, dotstyle=o, dotsize=0.14, 
			fillcolor=green](3.0914302,-3.0426745)
			\psdots[linecolor=black, dotstyle=o, dotsize=0.14, 
			fillcolor=green](6.387901,3.0052893)
			\psdots[linecolor=black, dotstyle=o, dotsize=0.14, 
			fillcolor=green](1.2443714,0.1102666)
			\psdots[linecolor=black, dotstyle=o, dotsize=0.14, 
			fillcolor=green](1.6502538,1.7043842)
			\pscircle[linecolor=colour1, linewidth=0.04, 
			dimen=outer](1.2479515,0.114567064){0.17919655}
			\psframe[linecolor=colour1, linewidth=0.04, 
			dimen=outer](3.3235612,0.99864167)(2.9924278,0.66750824)
			\rput[bl](4.5833573,1.1240096){0a}
			\rput[bl](5.242181,0.62400967){1a=0e}
			\rput[bl](4.559828,-1.5465785){3a}
			\rput[bl](3.2362983,-0.95834327){4a}
			\rput[bl](4.512769,3.7004802){0b}
			\rput[bl](6.359828,3.1416566){1b}
			\rput[bl](7.959828,1.7769508){2b=1e}
			\rput[bl](8.389239,-0.111284465){3b}
			\rput[bl](1.2186514,-2.0289316){8b}
			\rput[bl](2.7715926,-3.3877552){7b}
			\rput[bl](4.601004,-3.805402){6b}
			\rput[bl](6.5833573,-3.3524609){5b}
			\rput[bl](7.9068866,-1.8759904){4b}
			\rput[bl](2.2892396,3.2240098){11b}
			\end{pspicture}
			}
\caption[]{Sample Vortex Nerve with 6 filaments attached between vortexes}
\label{fig:bett8}
\end{figure}

\subsection{Cyclic groups for filament edges between vortexes.}$\mbox{}$\\
An edge attached between the inner and outer vortexes in a vortex nerve is called a filament, which has a cyclic group representation.   Let $e$ be a filament (denoted by $\fil e$) between a pair of vortexes.   A filament is bi-directional, {\em i.e.}, a filament can be traversed in either the forward (+e) or reverse direction (-e) relative to a starting vertex on the filament.   Hence, a filment is its own inverse and we write
\[
\overbrace{e + (-e) = e - e = 0.}^{\mbox{\textcolor{blue}{\bf no traversal of  filament $e$}}}
\]
Notice that traversal of $\fil e$ $k$ times is the same as traversing filament $e$ one time.   Hence,
\[
\overbrace{e +\cdots+ e = e + e = e.}^{\mbox{\textcolor{blue}{\bf filament $e$ is an identity element}}}
\]
Obviously, the traversal operation $+$ is Abelian, $e + e' =  e' + e$.   Consequently, a filament with the binary operation $+$ is an Abelian group, represented by $\left(\mbox{filament}\ e, +\right)$.    Every filament is its own generator.    A filament Abelian group is also written as $\left(\left\langle e\right\rangle, +\right)$.    

\begin{example}{\bf Multiple Filaments Attached Between Vortexes}\label{ex:filaments}.\\
Let a vortex nerve $\Nrv E$ contains 6 filaments attached between the vortexes $\left\langle a\right\rangle,\left\langle b\right\rangle$ is shown in Fig.~\ref{fig:bett8}, each with its own generator, namely,
\[
\overbrace{\left\langle e_1\right\rangle,\left\langle e_2\right\rangle,\dots,\left\langle e_6\right\rangle.}^{\mbox{\textcolor{blue}{\bf Six Filament Groups}}}
\]
Notice that each pair of vertices in $\Nrv E$ is path-connected.   This means that each member of the vortex nerve can be written as a linear combination of the generators.   For example, consider the pair of vertexes $2a, 0b$.  Then we have
\[
0b = \overbrace{2a + 0e + 1e + 1b + 0b.}^{\mbox{\textcolor{blue}{\bf Traverse edges starting at $2a$ and ending with $0b$}}}\mbox{\qquad \textcolor{blue}{\Squaresteel}}
\]
\end{example}

From Example~\ref{ex:filaments}, we have a way of representing a vortex nerve in terms of its generators.
That is, we can write
\[
\overbrace{\Nrv E = \left\{\left\langle a\right\rangle,\left\langle e_1\right\rangle,\left\langle e_2\right\rangle,\dots,\left\langle e_6\right\rangle, \left\langle b\right\rangle\right\}.}^{\mbox{\textcolor{blue}{\bf Vortex nerve is a collection of generators}}}
\]

 \begin{figure}[!ht]
	\centering
	 \includegraphics[width=115mm]{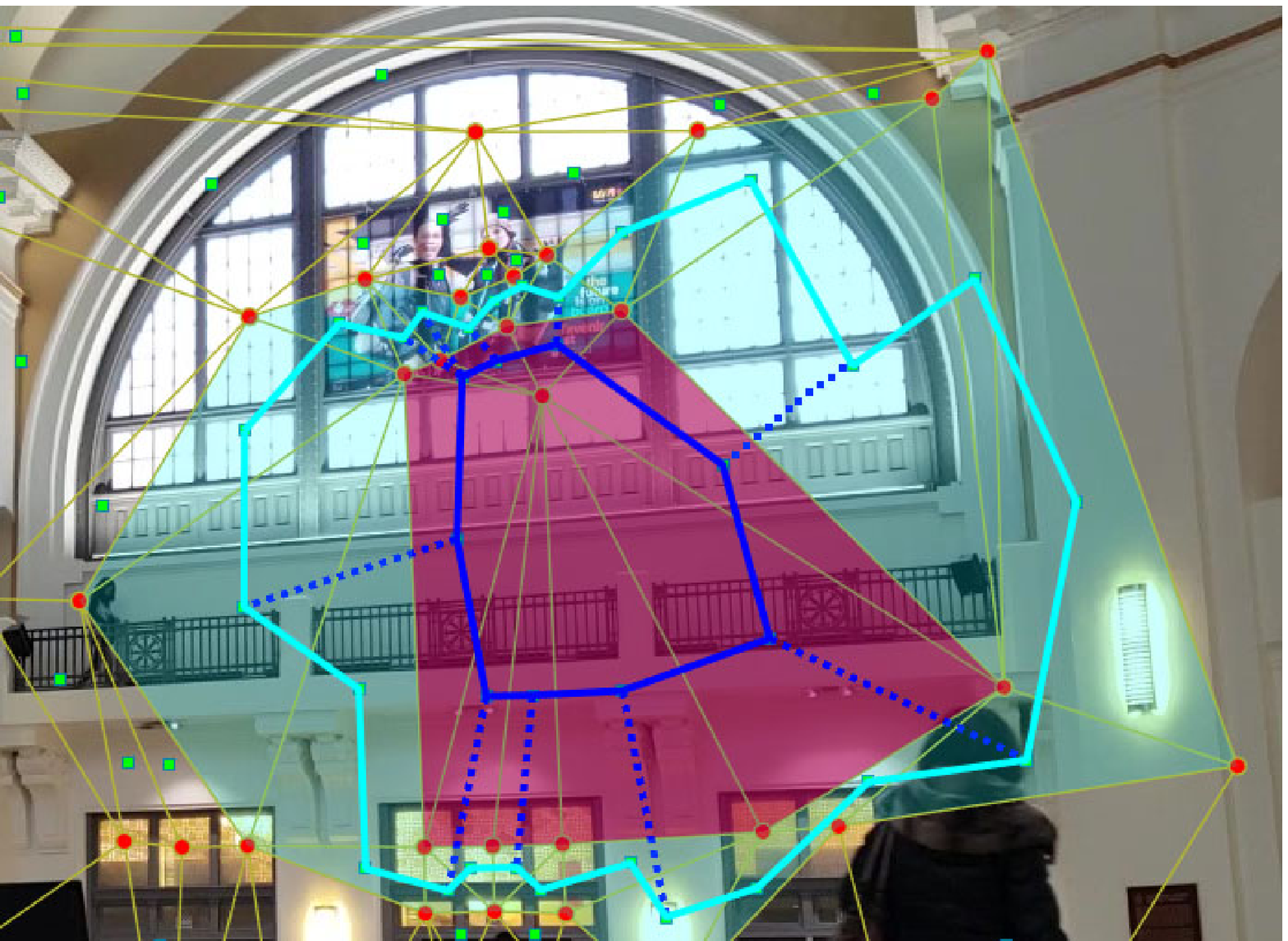}
	\caption{Barycentric vortex nerve on a triangulated video frame}
	\label{fig:vortexNerveFromArjuna}
\end{figure}

\subsection{Betti number for a vortex nerve on a triangulated video frame.}$\mbox{}$\\
Notice that a vortex nerve is a collection vortexes that are attached to each other.   This means that every pair of vertices in a vortex nerve is path-connected.   Also notice that each skeletal vortex in a vortex nerve is represented by a cyclic Abelian group with its own generator.   In effect, every vortex nerve has a free Abelian group representation.   A {\bf free abelian group} is an Abelian group with multiple generators, {\em i.e.}, every element of the group can be written as $\mathop{\sum}\limits_{i}g_ia$ for generators $\left\langle g_i \right\rangle$ in $G$.

A Betti number is a count of the number of generators (rank) in a free Abelian group~\cite[p. 151]{Goblin2010CUPhomology}.   This observation coupled with what we know about the cyclic Abelian group representation of each skeletal vortex and of each bi-directional sequence of edges attached between the vortexes, leads to the following result.

\begin{theorem}\label{thm:FundamentalOpticalVortexNerveTheorem}{\rm \cite[\S 4.13, p. 263]{Peters2020GTP}}.\\
Let $\mathcal{B}(\skCycNrv E)$ be the Betti number of $\skCycNrv E$.  A free Abelian group representation of $\skCycNrv E$ includes $k$ generators of cusp filament cyclic groups and two generators of the pair of the pair of cyclic groups representing the nesting, non-concentric nerve vortexes.   Hence, $\mathcal{B}(\skCycNrv E) = k + 2$.
\end{theorem}

\begin{example}
A sample optical vortex nerve  on the triangles of an Alexandroff nerve on a triangulated video frame, is shown in Fig.~\ref{fig:vortexNerveFromArjuna}.   Briefly, notice that there is a pair of nesting, non-concentric vortexes
with 9 attached between inner and outer vortex vertexes, each with its own generator.   Hence, from Theorem~\ref{thm:FundamentalOpticalVortexNerveTheorem}, the Betti number equals 2 + 10 = 12 for this sample nerve.
\qquad \textcolor{blue}{\Squaresteel}
\end{example}

\subsection{Maximal Nerve Complexes (MNCs)}
Of particular interest among all of the possible Alexandroff nerves on a triangulated video frame are those nerves that have a maximal number of triangles attached to a particular vertex.    In a triangulation of frame centroids, a \emph{maximal nerve complex} (MNC) has the highest number of centroids surrounding the common centroid at its center.   Each of the MNC vertexes is a centroid of an image dark region (hole).   
For this reason, an MNC has the highest number of dark regions (holes) in the triangulation of video frame centroids.   Also, the barycenters of the triangles on an MNC are between image holes, since each barycenter is in that part of a triangle between the centroids on frame dark regions.   Hence, connected barycenters model paths for light from either reflected or refracted light from visual scene surface shapes recorded in a video frame.   That is, the edge between a pair of barycenters stretches across a visual scene surface where there is reflected or refracted light.   Consequently, with a vortex nerve on an MNC, we will find the highest concentration of contrasting light and dark regions in an image.   It is well-known that a concentration of surface holes defines a surface shape.   In other words, a surface shape represented by an MNC will have the highest definition, {\em i.e.}, highest concentration of holes pinpointed by their centroids.

\begin{example}{\bf Sample Video Frame MNC}.\\
A sample video frame MNC is shown the majenta-coloured region in Fig.~\ref{fig:vortexNerveFromArjuna}.
\qquad \textcolor{blue}{\Squaresteel}
\end{example}

\begin{algorithm}[!ht]
\caption{\bf Video Barcode Construction Method}\label{algo:barcode}

	\SetKwData{matCentroids}{matCentroids}\SetKwData{triDelaunay}{triDelaunay}
	\SetKwData{mncNode}{mncNode}\SetKwData{BW}{BW}\SetKwData{matBarycenters}{matBarycenters}\SetKwData{frame}{frame}\SetKwData{bettyNo}{bettyNo}\SetKwData{kvortex}{kvortex}
	\SetKwData{mncBarycenters}{mncBarycenters}\SetKwData{count}{count}\SetKwData{i}{i}\SetKwData{j}{j}
	\SetKwData{spkComplex}{spkComplex}\SetKwData{filaments}{filaments}
	\SetKwFunction{plot}{plot}\SetKwFunction{FindCompress}{FindCompress}
	\SetKwFunction{length}{length}\SetKwFunction{polygon}{polygon}
	\SetKwFunction{skelSeedPoints}{skelSeedPoints}
	\SetKwInOut{Input}{input}\SetKwInOut{Output}{output}
	\Input{A video $Im$ of size $w\times l$ and $f$ frames}
	\Output{A triangulated video $w\times l$ and $f$ frames}
	\BlankLine
	
	\While{$count < f$}{
		\count $\leftarrow$ \emph{\count + 1}\;
		\emph{Read the video and get \frame{count}.}\;
		\matCentroids $\leftarrow$ \skelSeedPoints{\frame{count}} \;
		\triDelaunay $\leftarrow$ \emph{Perform Delaunay triangulation on 
		\matCentroids.}\;
		\plot{\triDelaunay}\;
		\mncNode  $\leftarrow$ \emph{Calculate the most common node in \triDelaunay}\;
		\matBarycenters  $\leftarrow$ \emph{Calculate barycenters of triangles in 
		\triDelaunay}\;
		\mncBarycenters $\leftarrow$ \emph{Select barycenters surrounding \mncNode from 
		\matBarycenters}\;
		\plot{\matBarycenters}\;
		\bettyNo (count)(j) =1\;
			\While{$i < \length{mncNode}$}{
			\kvortex(i)(1) $\leftarrow$ \mncBarycenters\;
			\j $\leftarrow$ \emph{1}\;
			\While{$\polygon{\spkComplex(i)(j)} encapsulate		
			\spkComplex(i)(j-1)$}{
				\kvortex(i)(j) $\leftarrow$ \emph{Calculate immediate neighboring 
				triangles}\;
				\spkComplex(i)(j) $\leftarrow$ \emph{Select barycenters of 
				\kvortex(i)(j) from \matBarycenters}\;
					\plot{\polygon{\spkComplex(i)(j)}}\;
					\filaments  $\leftarrow$ \emph{Connect each vertices of 
					\spkComplex(i)(j-1) to vertices in \spkComplex(i)(j) }\;
					\plot{\filaments}\;
					\bettyNo(count)(j) = \bettyNo(count)(j) + 
					\length{\spkComplex(i)(j-1)}+1\;
			}
		}
		\emph{Write \frame}\;
	}
	\plot{\bettyNo}\;
\end{algorithm}

\begin{figure}
	\centering
	\subfigure[Barcode for a triangulated video]{
		\includegraphics[width=\linewidth]{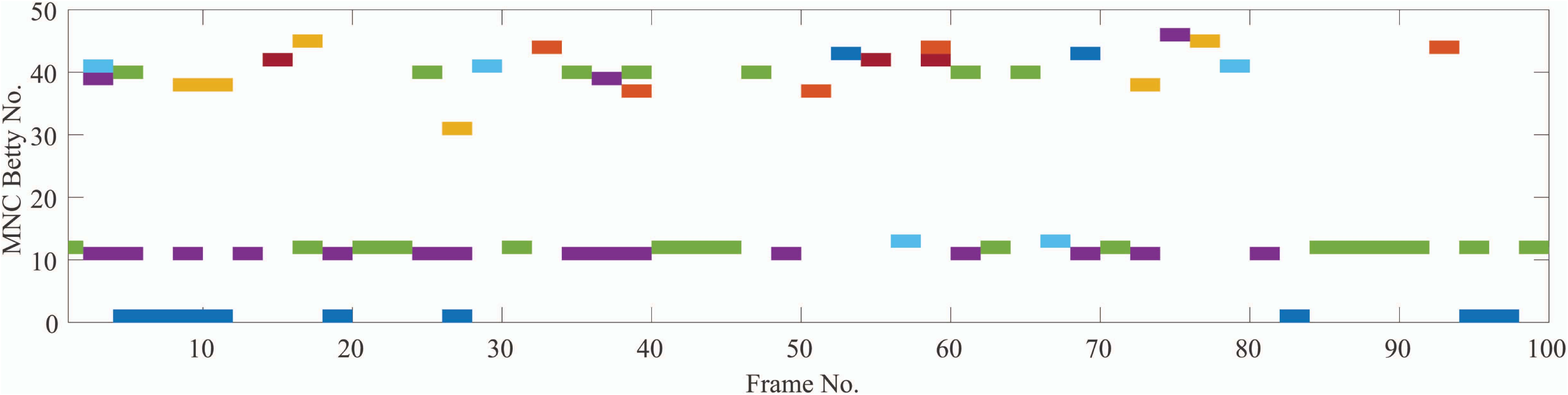}\label{fig:sample2barcode100}} 
	\subfigure[Barcode for a 2nd triangulated video]{
		\includegraphics[width=\linewidth]{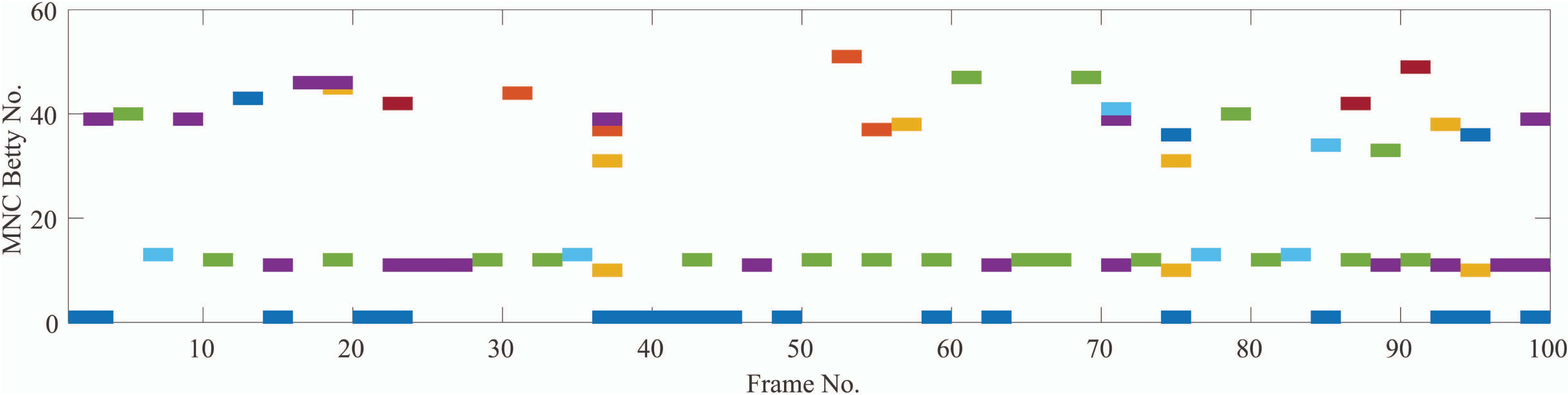} \label{fig:sample3barcode100}} 
	\caption{Sample Video Ghrist Barcodes for Two Videos}
	\label{fig:samples}
\end{figure}

\section{Betti Number-Based Video Barcode}
A Betti number-based video \emph{barcode} is a pictograph that records one or more occurrences of Betti numbers derived from vortex nerves across sequences of triangulated frames.   In our case, a frame Betti number tells us the number of generators in the frame vortex nerve.   A repetition of the same Betti number across a sequence of consecutive frames tells us that there is a similar shape outlined by a vortex nerve that recurs on the frames.

In the video barcode introduced in the paper, there is a 1-to-1 correspondence between a frame number and a Betti number.   There can be more than one vortex nerve on a frame.   Hence, a frame with more than one Betti number will result in the more than one bar in the same column of the video barcode.   The steps to construct a video barcode are given in Alg.~\ref{algo:barcode}.

\begin{example}{\bf Sample Video Barcodes}.\\
Sample video barcodes are given in Fig.~\ref{fig:sample2barcode100} and  Fig.~\ref{fig:sample3barcode100}.   A persistent video frame Betti number is represented by a row of contiguous bars.   
\qquad \textcolor{blue}{\Squaresteel}
\end{example}

The gaps between the sequences of contiguous bars are important.   Since each column in a video barcode corresponds to a video frame, a barcode row containing a sequence of contiguous bars corresponds to sequence of video frames that we can identify.   As a result, each row containing persistent Betti numbers leads to the production of a new video in which only frames containing vortex nerves with similar shapes appear in the video.  Each row video containing a persistent surface shape give us a closer look at the minute changes in surfaces covered by a vortex nerve with the same Betti number.    This leads to the following provable observations.

\begin{observation}
A Betti number-based video barcode row with no gaps indicates the presence of a vortex nerve with approximately the same shape in each video frame.
\qquad \textcolor{blue}{\Squaresteel}
\end{observation}

\begin{observation}
A Betti number-based video barcode row with large gaps between occurrences of a Betti number indicates the presence of dissimilar vortex nerve with dissimilar shapes in large number of video frames.
\qquad \textcolor{blue}{\Squaresteel}
\end{observation}

A pair of vortex nerves $\skCycNrv E, \skCycNrv E'$ are descriptively close, provided the vortex nerves have the same description, {\em i.e.}, a feature vector $\Phi(\skCycNrv E)$ that describes $\skCycNrv E$ will match the feature vector $\Phi(\skCycNrv E')$.    This is leads to the following observation.

\begin{observation}
A Betti number-based video barcode row with a persistent Betti number across the video frames implies the presence of descriptively close vortex nerves on the frames.
\qquad \textcolor{blue}{\Squaresteel}
\end{observation}

\section{Time complexity Analysis}
Fig.\ref*{fig:timeComplexity} shows the results of the time complexity analysis for the 
algorithm. In order to calculate the theoretical time complexity several assumptions 
were made. The time taken for built-in functions were not considered; for example time 
taken to import the video frame, save the triangulated frames, initializing variables, 
inner working of loops etc. Addition, Subtraction, Multiplication, Division were taken 
as 4 different calculations.Furthermore, allocation of values and array search function 
were considered to be one calculation.

The theoretically obtained time complexity in terms of big O notation was $mn^2$ where 
$m$ is the number of  MNCs and $n$ is the number of centroids. 

In order to obtain the actual time complexity plot, randomly generated points were 
used. This ensures that the generated points are not based on a particular image and 
will give more generic results. To plot the two graphs in the same plot, a scaling 
factor $k$ was calculated. So the final theoretical graph shown in 
Fig.\ref{fig:timeComplexity} is in the form of $kmn^2$. The value of $k$ was 
experimentally found to be $1.0526e^{-4} s$. 

The theoretical graph is shown in green dotted line. The actual time complexity is 
shown using the blue solid line. The number of  MNCs that were generated are shown by 
the gray stem plot. The red dotted line shows the $kn^2$ graph where the MNC number was 
not considered. It is evident from the plots that both number of centroids and the  
number of MNCs affects the time complexity of the algorithm. Furthermore, the $kmn^2$ 
graph follows the real time complexity graph very closely. 

\begin{figure}[!ht]
	\centering
	\includegraphics[width=\linewidth]{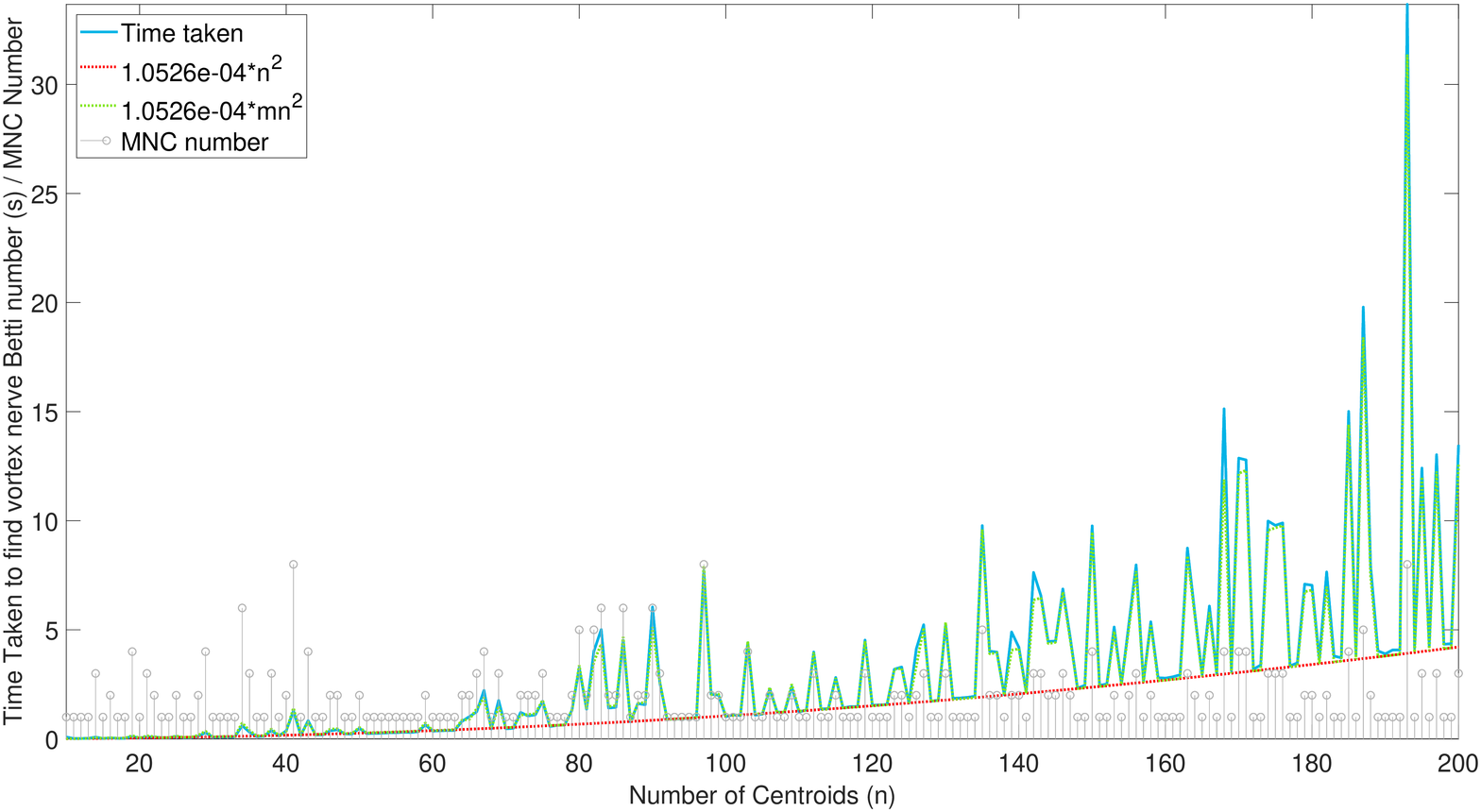}
	\caption{Time complexity.} 
	\label{fig:timeComplexity}
\end{figure}

\section{Conclusion}
A video, Betti nuumber-based form of Ghrist barcode has been introduced in this paper. 
This form of Ghrist barcode is useful in tracking the persistence of surface shapes recorded
sequences of video frames.     

\bibliographystyle{amsplain}
\bibliography{NSrefs}

\providecommand{\bysame}{\leavevmode\hbox to3em{\hrulefill}\thinspace}
\providecommand{\MR}{\relax\ifhmode\unskip\space\fi MR }
\providecommand{\MRhref}[2]{%
  \href{http://www.ams.org/mathscinet-getitem?mr=#1}{#2}
}
\providecommand{\href}[2]{#2}
\begin{thebibliography}{10}

\bibitem{AhmadPeters2018centroidalVortices}
M.Z. Ahmad and J.F. Peters, \emph{Maximal centroidal vortices in
  triangulations. a descriptive proximity framework in analyzing object
  shapes}, Theory and Applications of Math. \& Comp. Sci. \textbf{8} (2018),
  no.~1, 38--59, ISSN 2067-6202.

\bibitem{Alexandroff1926MAnnNerfTheorem}
P.~Alexandroff, \emph{Simpliziale approximationen in der allgemeinen
  topologie}, Mathematische Annalen \textbf{101} (1926), no.~1, 452--456,
  MR1512546.

\bibitem{Alexandroff1932elementaryConcepts}
\bysame, \emph{Elementary concepts of topology}, Dover Publications, Inc., New
  York, 1965, 63 pp., translation of Einfachste Grundbegriffe der Topologie
  [Springer, Berlin, 1932], translated by Alan E. Farley , Preface by D.
  Hilbert, MR0149463.

\bibitem{AlexandroffHopf1935Topologie}
P.~Alexandroff and H.~Hopf, \emph{Topologie. {B}and i}, Springer, Berlin, 1935,
  Zbl 13, 79; reprinted Chelsea Publishing Co., Bronx, N. Y., 1972. iii+637
  pp., MR0345087.

\bibitem{Alexandroff1956combinatorialTopology}
P.S. Alexandrov, \emph{Combinatorial topology}, Graylock Press, Baltimore, Md,
  USA, 1956, {x}vi+ 244 pp. ISBN: 0-486-40179-0.

\bibitem{Edelsbrunner1999}
H.~Edelsbrunner and J.L. Harer, \emph{Computational topology. an introduction},
  Amer. Math. Soc., Providence, RI, 2010, {x}ii+241 pp. ISBN:
  978-0-8218-4925-5, MR2572029.

\bibitem{EdelsbrunnerLetscherZomorodian2000IEEEbarcode}
H.~Edelsbrunner, D.~Letscher, and A.~Zomorodian, \emph{Topological persistence
  and simplification}, 41st Annual Symposium on Foundations of Computer
  Science, IEEE Comput. Soc. Press, Los Alamitos, California, 2000, MR1931842,
  pp.~454--463.

\bibitem{EdelsbrunnerLetscherZomorodian2001DCGbarcode}
\bysame, \emph{Topological persistence and simplification}, Discrete Comput.
  Geom. \textbf{28} (2001), no.~4, 511--533, MR1949898.

\bibitem{Ghrist2008BAMSbarcodePersistence}
R.~Ghrist, \emph{Barcodes: the persistent topology of data}, Bull. Amer. Math.
  Soc. (N.S.) \textbf{45} (2008), no.~1, 61--75, MR2358377.

\bibitem{Ghrist2014elementaryAppliedGeometry}
R.W. Ghrist, \emph{Elementary applied topology}, University of Pennsylvania,
  2014, vi+269 pp. ISBN: 978-1-5028-8085-7.

\bibitem{Goblin2010CUPhomology}
P.~Giblin, \emph{Graphs, surfaces and homology, 3rd ed.}, Cambridge University
  Press, Cambridge, GB, 2016, xx+251 pp. ISBN: 978-0-521-15405-5, MR2722281,
  first edition in 1981, MR0643363.

\bibitem{LeNguyenTran2014videoFrameWatermarking}
B.~Le, H.~Nguyen, and D.~Tran, \emph{A robust fingerprint watermark-based
  authentication scheme in h.264/avc video}, Vietnam J Comput Sci \textbf{1}
  (2014), 193--206, DOI: \url{10.1007/s40595-014-0021-x}.

\bibitem{Perea2018arXivhomologyBarcode}
J.A. Perea, \emph{A brief history of persistence}, arXiv \textbf{1809} (2018),
  no.~036249, 1--11.

\bibitem{Peters2017AMSJshapeSignature}
J.F. Peters, \emph{Proximal planar shape signatures. {H}omology nerves and
  descriptive proximity}, Advan. in Math: Sci. J \textbf{6} (2017), no.~2,
  71--85, Zbl 06855051.

\bibitem{Peters2018AlMSproximalPlanarShapes}
\bysame, \emph{Proximal planar shapes. correspondence between triangulated
  shapes and nerve complexes}, Bulletin of the Allahabad Mathematical Society
  \textbf{33} (2018), 113--137, MR3793556, Zbl 06937935, Review by D. Leseberg
  (Berlin).

\bibitem{Peters2018JMSMvortexNerves}
\bysame, \emph{Proximal vortex cycles and vortex nerve structures.
  non-concentric, nesting, possibly overlapping homology cell complexes},
  Journal of Mathematical Sciences and Modelling \textbf{1} (2018), no.~2,
  56--72, ISSN 2636-8692, \url{www.dergipark.gov.tr/jmsm}, See, also,
  \url{https://arxiv.org/abs/1805.03998}.

\bibitem{Peters2020GTP}
\bysame, \emph{Computational geometry, topology and physics of digital images
  with applications. {S}hape complexes, optical vortex nerves and proximities},
  Springer Int. Pub. AG, Cham, Switzerland, 2020, vii+563, \emph{in press}.

\bibitem{PetersRamanna2018shapeDescriptions}
J.F. Peters and S.~Ramanna, \emph{Shape descriptions and classes of shapes. {A}
  proximal physical geometry approach}, Advances in feature selection for data
  and pattern recognition (B.~Zielosko U.~Sta\'{n}czyk and L.C. Jain, eds.),
  Springer, 2018, MR3895981, pp.~203--225.

\end{thebibliography}

\end{document}